\documentclass[reqno,10pt,centertags]{amsart}

\usepackage{amsmath,amssymb,amsthm,graphicx,mathrsfs,url}
\usepackage{mathtools}
\usepackage{relsize}
\usepackage{braket}
\usepackage{color}
\usepackage{hyperref}
\numberwithin{equation}{section}

\newtheorem{thm}{Theorem}[section]
\newtheorem{cor}[thm]{Corollary}
\newtheorem{lem}[thm]{Lemma}
\newtheorem{prop}[thm]{Proposition}

\theoremstyle{definition}
\newtheorem{defn}[thm]{Definition}
\newtheorem{rem}[thm]{Remark}

\newcommand{\dC}{{\mathbb{C}}}

\newcommand{\dR}{{\mathbb{R}}}

\newcommand{\cF}{{\mathcal F}}
\newcommand{\cG}{{\mathcal G}}
\newcommand{\cH}{{\mathcal H}}

\renewcommand{\ker}{\operatorname{ker}}
\newcommand{\ran}{\operatorname{ran}}
\newcommand{\dom}{\operatorname{dom}}

\newcommand{\sig}{\operatorname{sgn}}

\begin{document}

\title[QBTs, self-adjoint extensions, and Robin Laplacians]{Quasi boundary triples, self-adjoint extensions, and Robin Laplacians on the half-space}

\author{Jussi Behrndt}
\address{Institut f\"ur Angewandte Mathematik, Technische Universit\"at Graz, Steyrergasse 30, 8010 Graz, Austria}
\email{behrndt@tugraz.at}
\urladdr{http://www.math.tugraz.at/$\sim$behrndt/}

\author{Peter Schlosser}
\address{Institut f\"ur Angewandte Mathematik, Technische Universit\"at Graz, Steyrergasse 30, 8010 Graz, Austria}
\email{schlosser@tugraz.at}
\urladdr{}

\begin{abstract}
In this note self-adjoint extensions of symmetric operators are
investigated by using the abstract technique of quasi boundary triples and
their Weyl functions. The main result is an extension of \cite[Theorem~2.6]{BeLaLoRo2017_1} which provides 
sufficient conditions on the parameter in the boundary space to induce self-adjoint realizations.
As an example self-adjoint Robin Laplacians on the half-space with boundary conditions involving an unbounded coefficient
are considered.
\end{abstract}



\maketitle

\section{Introduction}

The concept of quasi boundary triples and their Weyl functions is a useful tool in the spectral theory of symmetric and self-adjoint 
elliptic partial differential operators. This abstract notion from \cite{BL07,BL12} 
is a slight generalization of ordinary boundary triples and their Weyl functions from \cite{BGP08,DM91,GG91},
adapted and extended in such a way that it directly applies to elliptic boundary value problems in the Hilbert space framework.

Very roughly speaking, 
a quasi boundary triple consists of a boundary Hilbert space $\cG$ -- in applications typically the $L^2$-space on 
the boundary of some domain $\Omega$ -- and two boundary mappings $\Gamma_0$ and $\Gamma_1$ 
that satisfy an abstract second Green identity. A natural choice are the Neumann and Dirichlet trace operators if one deals with the Laplacian in $L^2(\Omega)$.
The boundary mappings are defined on the domain of some operator $T$ which is a core of the maximal operator; in the case of the Laplacian, the core 
$H^2(\Omega)$ is often a convenient choice.
The Weyl function corresponding to a quasi boundary triple 
can be viewed as the abstract counterpart of the Dirichlet-to-Neumann map and is an important analytic object
since it can be used to characterize the spectrum of the self-adjoint realizations in this theory. One uses abstract boundary conditions to define restrictions of $T$ in the form
\begin{equation*}
 A_{[B]}f=Tf,\qquad \dom A_{[B]}=\bigl\{f\in\dom T:\Gamma_0 f=B\Gamma_1 f\bigr\},
\end{equation*}
where $B$ is an operator in the boundary space $\cG$. It is an immediate consequence 
of the abstract second Green identity that a symmetric operator $B$ leads to a symmetric operator $A_{[B]}$, but in general a self-adjoint boundary parameter 
$B$ does not induce a self-adjoint operator $A_{[B]}$ -- a fact that is not too surprising when taking into account that 
the range of the boundary mappings is not necessarily the whole boundary space $\cG$; cf. Definition~\ref{defi_Quasi_boundary_triple}.

It is one of the main objectives of the present note to provide a new useful sufficient condition on the boundary parameter $B$ and the properties of the Weyl function
to ensure self-adjointness of the extension $A_{[B]}$. Here we generalize a recent result from \cite{BeLaLoRo2017_1} by allowing boundary operators $B$ that are factorized in 
the form $B=B_1B_2$, or more general $B\subset B_1B_2$. The assumptions on $B$ in \cite[Theorem~2.6]{BeLaLoRo2017_1} are here replaced by similar ones on $B_1$ and $B_2$.
We refer the reader to Theorem~\ref{satz_Selfadjoint_boundary_operators} and the discussion afterwards for more details.

As an example and illustration for the abstract techniques we discuss the Laplacian on the half-space $\dR^d_+=\{x\in\dR^d : x_d>0\}$ 
in any dimension $d\geq 2$ in Section~\ref{sec3}. The key feature is that Theorem~\ref{satz_Selfadjoint_boundary_operators}
and Corollary~\ref{cor2_Selfadjoint_boundary_operators} can be applied for the Laplace operator with Robin 
boundary conditions $\tau_N f=\alpha\tau_D f$ on $\partial\dR^d_+\simeq\dR^{d-1}$, where
$$\alpha\in L^p(\dR^{d-1})+L^\infty(\dR^{d-1})$$ is real-valued. 
In Theorem~\ref{satz_Representation_via_boundary_values} we have the slightly stronger 
assumption $p>\frac{4}{3}(d-1)$ if $d\geq 3$ and $p>2$ if $d=2$ than the usual form method requires (namely, $p=d-1$ if $d\geq 3$ and $p>1$ if $d=2$  
is sufficient by Proposition \ref{prop_Operator_via_Quadratic_form}), but also at the same time a higher Sobolev regularity for the operator domain.
For other related variants of Theorem~\ref{satz_Representation_via_boundary_values} we also refer the reader to \cite[Theorem 7.2]{AGW14} which 
provides $H^2$-regularity for more general
second order elliptic differential expressions on  
certain unbounded non-smooth domains (see Remark~\ref{agwrem}), to
\cite[Theorem 4.5 and Lemma 5.3]{GM09} for the case of Laplacians on bounded Lipschitz domains, 
and to \cite[Section~2]{FL08}.
In this context we also mention the 
contributions \cite{AGW14,B09,BD08,ES88,MR09,NP17} dealing with Robin Laplacians with singular boundary conditions and we refer to 
\cite{AW03,BPP18,BP16,D00,G11,HK18,HKR17,HP15,KL12,LP08,LR12,PP15,PP16,R14} for some other recent works
on spectral problems for Robin Laplacians.

\section*{Acknowledgements}
Jussi Behrndt is most grateful 
to the organizers of the conference {\it Mathematics, Signal Processing and Linear Systems:
New Problems and Directions}, November 2017, at 
Chapman University, Orange, California, for creating a stimulating atmosphere and for the great hospitality.
The authors also wish to thank Vladimir Lotoreichik, Konstantin Pankrashkin, Nicolas Popoff, and Jonathan Rohleder for helpful comments and remarks.
Both authors gratefully acknowledge financial support
by the Austrian Science Fund (FWF), project P~25162-N26,
and the Austria-Czech Republic cooperation grant  CZ02/2017 by the 
Austrian Agency for International Cooperation in Education and Research (OeAD). 

\section{Quasi boundary triples and self-adjoint extensions}

In this section we first recall the notion of quasi boundary triples and their Weyl functions in the extension theory of symmetric operators from \cite{BL07,BL12}. Afterwards we provide a new sufficient
criterion for self-adjointness in Theorem~\ref{satz_Selfadjoint_boundary_operators}, which is the main abstract result in this note. 

In the following let $\cH$ be a Hilbert space with inner product $(\cdot,\cdot)_\cH$. 
The next definition is a generalization of the concept of ordinary and generalized boundary triples; cf. \cite{BGP08,DHMS06,DM91,DM95,GG91}. 

\begin{defn}\label{defi_Quasi_boundary_triple}
Let $S$ be a densely defined, closed, symmetric operator in $\cH$ and let $T$ be a closable operator with $\overline{T}=S^*$.
A triple $\{\cG,\Gamma_0,\Gamma_1\}$ is a {\em quasi boundary triple} for $T\subset S^*$ if
$(\cG,(\cdot,\cdot)_\cG)$ is a Hilbert space and the linear mappings $\Gamma_0,\Gamma_1:\dom T\rightarrow\cG$ satisfy the following conditions (i)--(iii).
\begin{enumerate}
\item [(i)] The abstract second Green identity 
\begin{equation}\label{Eq_Abstract_Greens_identity}
 (Tf,g)_\cH-(f,Tg)_\cH=(\Gamma_1 f,\Gamma_0 g)_\cG-(\Gamma_0 f,\Gamma_1 g)_\cG
\end{equation}
holds for all $f,g\in\dom T$.
\item [(ii)] The range of $(\Gamma_0,\Gamma_1)^\top:\dom T\rightarrow\cG\times\cG$ is dense.
\item[(iii)] The operator $A_0:= T\upharpoonright\ker\Gamma_0$ is self-adjoint in $\cH$.
\end{enumerate}
\end{defn}

Recall from \cite{BL07,BL12} that for a densely defined, closed, symmetric operator $S$ in $\cH$ a quasi boundary triple $\{\cG,\Gamma_0,\Gamma_1\}$ 
exists if and only if the deficiency indices of $S$ coincide. In this case one has $\dom S=\ker\Gamma_0\cap\ker\Gamma_1$. The notion of quasi boundary triples 
reduces to the well-known concept of ordinary boundary triples if $T=S^*$. For more details we refer the reader to \cite{BL07,BL12}.

Assume now that $\{\cG,\Gamma_0,\Gamma_1\}$ is a quasi boundary triple for $T\subset S^*$.
In a similar way as for ordinary and generalized boundary triples in \cite{DM91,DM95}
one associates the $\gamma$-field and the Weyl function. 
Their definition and some of their properties will now be recalled very briefly. Again we refer the reader to \cite{BL07,BL12} for a more detailed exposition.
Observe first that the direct sum decomposition
\begin{equation}\label{Eq_DomT_decomposition}
 \dom T=\dom A_0\,\dot+\,\ker(T-\lambda)=\ker\Gamma_0\,\dot+\,\ker(T-\lambda),\quad \lambda\in\rho(A_0),
\end{equation}
implies that $\Gamma_0\upharpoonright\ker(T-\lambda)$ is invertible for $\lambda\in\rho(A_0)$. The $\gamma$-field $\gamma$ 
and Weyl function $M$ are
then defined as operator-valued functions on $\rho(A_0)$ by 
\begin{equation}\label{Eq_Gamma_field}
\lambda\mapsto\gamma(\lambda):=\bigl(\Gamma_0\upharpoonright\ker(T-\lambda)\bigr)^{-1}\quad\text{and}\quad \lambda\mapsto M(\lambda):=\Gamma_1\gamma(\lambda),
\end{equation}
respectively. It is clear from \eqref{Eq_DomT_decomposition} that $\dom\gamma(\lambda)=\dom M(\lambda)=\ran\Gamma_0$ for all $\lambda\in\rho(A_0)$. Moreover,
the values $\gamma(\lambda)$ of the $\gamma$-field are densely defined and bounded operators from $\cG$ into $\cH$ such that $\ran \gamma(\lambda)=\ker(T-\lambda)$.
With the help of the abstract second Green identity in (\ref{Eq_Abstract_Greens_identity}) one verifies the representation 
\begin{equation}\label{gam11}
\gamma(\overline\lambda)^*=\Gamma_1(A_0-\lambda)^{-1},\qquad \lambda\in\rho(A_0),
\end{equation}
of the adjoint $\gamma$-field, which is a bounded and everywhere defined operator from $\cH$ into $\cG$. The values $M(\lambda)$ of the Weyl function 
are operators in $\cG$ which are not necessarily closed
and in general unbounded. Note that also $\ran M(\lambda)\subset\ran\Gamma_1$ by definition.

For a given quasi boundary triple $\{\cG,\Gamma_0,\Gamma_1\}$ and an operator $B$ in $\cG$ the extension $A_{[B]}$ of $S$ in $\cH$ is defined as
\begin{equation}\label{Eq_Boundary_operator}
A_{[B]} f=Tf,\qquad \dom A_{[B]}=\bigl\{f\in\dom T: \Gamma_0f= B\Gamma_1f\bigr\}. 
\end{equation}
In contrast to ordinary boundary triples (see \cite{BGP08,DM91,GG91}) a self-adjoint boundary operator $B$ in $\cG$ does not necessarily induce a self-adjoint 
extension $A_{[B]}$ in $\cH$. There are various results in the literature that provide sufficient conditions for this conclusion to hold, 
see, e.g., \cite{BL07,BL12,BeLaLoRo2017_1}.
Our aim in the next theorem is to provide a useful generalization of a recent result in \cite{BeLaLoRo2017_1}; cf. Corollary~\ref{cor_Selfadjoint_boundary_operators}.

\begin{thm}\label{satz_Selfadjoint_boundary_operators}
Let $S$ be a densely defined, closed, symmetric operator in $\cH$ and let $\{\cG,\Gamma_0,\Gamma_1\}$ be a quasi boundary triple for $T\subset S^*$
with $A_0=T\upharpoonright\ker\Gamma_0$, $\gamma$-field $\gamma$ and Weyl function $M$. Let $\lambda_0\in\rho(A_0)\cap\dR$ and let $B$ be a symmetric operator
in $\cG$. Assume that $B\subset B_1B_2$ holds with some operators $B_1,B_2$ in $\cG$ and that the following conditions are satisfied.
\begin{enumerate}
\item[{\rm (i)}] $1\in\rho(B_2\overline{M(\lambda_0)B_1})$;
\item[{\rm (ii)}] $\ran(B_2\overline{M(\lambda_0)B_1})\subset\ran\Gamma_0\cap\dom B_1$;
\item[{\rm (iii)}] $\ran(B_1\upharpoonright\ran\Gamma_0)\subset\ran\Gamma_0$;
\item[{\rm (iv)}] $\ran(B_2\upharpoonright\ran\Gamma_1)\subset\ran\Gamma_0$; 
\item[{\rm (v)}] $\ran\Gamma_1\subset\dom B$.
\end{enumerate}
Then the extension $A_{[B]}$ in \eqref{Eq_Boundary_operator}
is a self-adjoint operator in $\cH$ and for every $\lambda\in\rho(A_0)\cap\rho(A_{[B]})$ the Krein type resolvent formula
\begin{equation}\label{Eq_Krein_formula}
(A_{[B]}-\lambda)^{-1}-(A_0-\lambda)^{-1}=\gamma(\lambda)B_1(1-B_2M(\lambda)B_1)^{-1}B_2\gamma(\overline{\lambda})^*
\end{equation}
is valid.
\end{thm}

In the next corollary the special case that $\ran\Gamma_0=\cG$ is formulated. In this situation the quasi boundary triple $\{\cG,\Gamma_0,\Gamma_1\}$ is a
generalized boundary triple in the sense of \cite{DHMS06,DM95}. It is clear that condition (ii) 
in Theorem~\ref{satz_Selfadjoint_boundary_operators} simplifies and that conditions (iii) and (iv) are automatically satisfied in this case.

\begin{cor}\label{cor2_Selfadjoint_boundary_operators}
Let $\{\cG,\Gamma_0,\Gamma_1\}$, $A_0=T\upharpoonright\ker\Gamma_0$, $M$ and $\lambda_0\in\rho(A_0)\cap\dR$, and $B\subset B_1B_2$ be as in 
Theorem~\ref{satz_Selfadjoint_boundary_operators}. Assume, in addition, that $\ran\Gamma_0=\cG$ and that the following conditions are satisfied.
\begin{enumerate}
\item[{\rm (i)}] $1\in\rho(B_2\overline{M(\lambda_0)B_1})$;
\item[{\rm (ii)}] $\ran(B_2\overline{M(\lambda_0)B_1})\subset\dom B_1$;
\item[{\rm (iii)}] $\ran\Gamma_1\subset\dom B$.
\end{enumerate}
Then the extension $A_{[B]}$ in \eqref{Eq_Boundary_operator}
is a self-adjoint operator in $\cH$ and for every $\lambda\in\rho(A_0)\cap\rho(A_{[B]})$ the Krein type resolvent formula
\eqref{Eq_Krein_formula}
is valid.
\end{cor}

The next corollary shows that for the special choice $B_1=I_\cG$ and $B_2=B$  
Theorem \ref{satz_Selfadjoint_boundary_operators} coincides with \cite[Theorem 2.6]{BeLaLoRo2017_1}.

\begin{cor}\label{cor_Selfadjoint_boundary_operators}
Let $\{\cG,\Gamma_0,\Gamma_1\}$, $A_0=T\upharpoonright\ker\Gamma_0$, $M$, $\gamma$ and $\lambda_0\in\rho(A_0)\cap\dR$ be as in Theorem~\ref{satz_Selfadjoint_boundary_operators} and assume that $B$ is a symmetric operator in $\cG$ such that 
the following conditions are satisfied.
\begin{enumerate}
\item[{\rm (i)}] $1\in\rho(B\overline{M(\lambda_0)})$;
\item[{\rm (ii)}] $\ran(B\overline{M(\lambda_0)})\subset\ran\Gamma_0$;
\item[{\rm (iii)}] $\ran(B\upharpoonright\ran\Gamma_1)\subset\ran\Gamma_0$; 
\item[{\rm (iv)}] $\ran\Gamma_1\subset\dom B$.
\end{enumerate}
Then the extension $A_{[B]}$ in \eqref{Eq_Boundary_operator}
is a self-adjoint operator in $\cH$ and for every $\lambda\in\rho(A_0)\cap\rho(A_{[B]})$ the Krein type resolvent formula
\begin{equation*}
(A_{[B]}-\lambda)^{-1}-(A_0-\lambda)^{-1}=\gamma(\lambda)(1-BM(\lambda))^{-1}B\gamma(\overline{\lambda})^*
\end{equation*}
is valid.
\end{cor}

\begin{rem}
 The assumption $\lambda_0\in\rho(A_0)\cap\dR$ and conditions (i)-(ii) in Theorem~\ref{satz_Selfadjoint_boundary_operators} (and similarly in Corollary~\ref{cor2_Selfadjoint_boundary_operators} and Corollary~\ref{cor_Selfadjoint_boundary_operators}) can be replaced by assuming that there exist $\lambda_\pm\in\dC^\pm$ with the properties
\begin{enumerate}
\item[{\rm (i')}] $1\in\rho(B_2\overline{M(\lambda_\pm)B_1})$;
\item[{\rm (ii')}] $\ran(B_2\overline{M(\lambda_\pm)B_1})\subset\ran\Gamma_0\cap\dom B_1$.
\end{enumerate}
\end{rem}

\begin{proof}[Proof of Theorem \ref{satz_Selfadjoint_boundary_operators}]
The proof is split into four separate steps: First the self-adjointness of $A_{[B]}$ is shown in Steps 1 and 2 and afterwards, in Step 3 and 4, the resolvent formula (\ref{Eq_Krein_formula}) is verified.

\vspace{0.2cm}
\noindent
\textit{Step 1.} In this step we prove the inclusion
\begin{equation}\label{Eq_Step_1}
\ran(B_2\gamma(\lambda_0)^*)\subset\ran(1-B_2M(\lambda_0)B_1).
\end{equation}
Let $\psi\in\ran(B_2\gamma(\lambda_0)^*)$. Then \eqref{gam11}, (iv)--(v) and $B\subset B_1B_2$ yield
\begin{equation}\label{Eq_Selfadjoint_boundary_operators_8}
\psi\in\ran (B_2\upharpoonright\ran\Gamma_1)\subset\ran\Gamma_0\cap\dom B_1.
\end{equation}
Consider $\varphi:=(1-B_2\overline{M(\lambda_0)B_1})^{-1}\psi$, which is well-defined by (i) and observe that
\begin{equation}\label{Eq_Selfadjoint_boundary_operators_1}
\varphi-\psi=B_2\overline{M(\lambda_0)B_1}\varphi\in\ran\Gamma_0\cap\dom B_1
\end{equation}
by (ii). Combining (\ref{Eq_Selfadjoint_boundary_operators_8})--(\ref{Eq_Selfadjoint_boundary_operators_1}) we conclude $\varphi\in\ran\Gamma_0\cap\dom B_1$ and now 
(iii) shows $B_1\varphi\in\ran\Gamma_0=\dom M(\lambda_0)$. Therefore (\ref{Eq_Selfadjoint_boundary_operators_1}) can be written as
\begin{equation*}
(1-B_2M(\lambda_0)B_1)\varphi=\psi,
\end{equation*}
and hence (\ref{Eq_Step_1}) holds.

\vspace{0.2cm}
\noindent
\textit{Step 2.} We will now prove that the operator $A_{[B]}$ in \eqref{Eq_Boundary_operator} is self-adjoint in $\cH$. Note first that for $f,g\in\dom A_{[B]}$ one has 
\begin{equation*}
\begin{split}
(A_{[B]}f,g)_\cH-(f,A_{[B]}g)_\cH&=(Tf,g)_\cH-(f,Tg)_\cH\\
&=(\Gamma_1 f,\Gamma_0 g)_\cG-(\Gamma_0 f,\Gamma_1 g)_\cG\\
&=(\Gamma_1 f,B\Gamma_1 g)_\cG-(B \Gamma_1 f,\Gamma_1 g)_\cG\\
&=0
\end{split}
\end{equation*}
by the abstract second Green identity (\ref{Eq_Abstract_Greens_identity}) and the symmetry of $B$ in $\cG$. Therefore $A_{[B]}$ is symmetric in $\cH$ and hence it suffices to show that
\begin{equation}\label{Eq_Step_2}
\ran(A_{[B]}-\lambda_0)=\cH.
\end{equation}
 
Fix $h\in\cH$. By \eqref{gam11}, (v) and $B\subset B_1B_2$, the element $B_2\gamma(\lambda_0)^*h$ is well-defined and according to (\ref{Eq_Step_1}) there exists some $g\in\dom(B_2M(\lambda_0)B_1)$ such that
\begin{equation}\label{Eq_Selfadjoint_boundary_operators_2}
B_2\gamma(\lambda_0)^*h=(1-B_2M(\lambda_0)B_1)g.
\end{equation}
Consider
\begin{equation*}
f:=(A_0-\lambda_0)^{-1}h+\gamma(\lambda_0)B_1g
\end{equation*}
and note that $(T-\lambda_0) f=h $ since $\ran\gamma(\lambda_0)=\ker(T-\lambda_0)$; cf. \eqref{Eq_Gamma_field}.  We claim that $f\in\dom A_{[B]}$. In fact, since $\dom A_0=\ker\Gamma_0$
it follows from \eqref{gam11}, the definition of the $\gamma$-field and the Weyl function that
\begin{equation}\label{Eq_Selfadjoint_boundary_operators_5}
\Gamma_0 f=B_1g\quad\text{and}\quad \Gamma_1 f=\gamma(\lambda_0)^*h+M(\lambda_0)B_1g.
\end{equation}
Making use of condition (v) and $B\subset B_1B_2$ we then conclude 
\begin{equation*}
B\Gamma_1 f=B_1\bigl(B_2\gamma(\lambda_0)^*h+B_2M(\lambda_0)B_1g\bigr)=B_1 g=\Gamma_0 f
\end{equation*}
from \eqref{Eq_Selfadjoint_boundary_operators_2} and \eqref{Eq_Selfadjoint_boundary_operators_5}. Hence $f\in\dom A_{[B]}$ and $(A_{[B]}-\lambda_0)f=(T-\lambda_0) f=h $. Thus, \eqref{Eq_Step_2} holds and therefore $A_{[B]}$ is self-adjoint in $\cH$.

\vspace{0.2cm}
\noindent
\textit{Step 3.} In this step we show that 
\begin{equation}\label{Eq_Step_3}
\ker (1-B_2M(\lambda)B_1)=\{0\},\qquad \lambda\in\rho(A_0)\cap\rho(A_{[B]}).
\end{equation}
In fact, for $\varphi\in\ker(1-B_2M(\lambda)B_1)$ one has $\varphi=B_2M(\lambda)B_1\varphi\in\ran\Gamma_0$ by (iv) and $\ran M(\lambda)\subset\ran\Gamma_1$. Making use of 
(iii) we find
\begin{equation}\label{Eq_Step_4}
B_1\varphi=B_1B_2M(\lambda)B_1\varphi\in\ran\Gamma_0.
\end{equation}
Using the definition of the $\gamma$-field and the Weyl function and (v) we can rewrite \eqref{Eq_Step_4} in the form
\begin{equation*}
\Gamma_0\gamma(\lambda)B_1\varphi=B\Gamma_1\gamma(\lambda)B_1\varphi,
\end{equation*}
which shows that $\gamma(\lambda)B_1\varphi\in\dom A_{[B]}$. Since $\ran\gamma(\lambda)=\ker(T-\lambda)$ and $\lambda\in\rho(A_{[B]})$ we conclude
\begin{equation*}
\gamma(\lambda)B_1\varphi\in\ker(A_{[B]}-\lambda)=\{0\},
\end{equation*}
and hence $\varphi=B_2\Gamma_1\gamma(\lambda)B_1\varphi=0$. We have shown (\ref{Eq_Step_3}).

\vspace{0.2cm}
\noindent
\textit{Step 4.} For $\lambda\in\rho(A_0)\cap\rho(A_{[B]})$ we prove $\ran(B_2\gamma(\overline{\lambda})^*)\subset\ran(1-B_2M(\lambda)B_1)$ and the resolvent formula (\ref{Eq_Krein_formula}).
For $h\in \cH$ define 
\begin{equation}\label{oho}
f_B:=(A_{[B]}-\lambda)^{-1}h\quad\text{and}\quad f_0:=(A_0-\lambda)^{-1}h.
\end{equation}
Then we have $f_B-f_0\in\ker(T-\lambda)$ and hence
\begin{equation}\label{Eq_Krein_formula_2}
\gamma(\lambda)\Gamma_0(f_B-f_0)=f_B-f_0.
\end{equation}
Furthermore, the definitions of $A_0$, $A_{[B]}$ and \eqref{gam11} show
\begin{equation}\label{Eq_Krein_formula_3}
\Gamma_0 f_0=0,\quad \Gamma_1 f_0=\gamma(\overline{\lambda})^*h,\quad\text{and}\quad\Gamma_0f_B=B\Gamma_1f_B.
\end{equation}
The element $B_2M(\lambda)B_1B_2\Gamma_1f_B$ is well-defined by (iii)--(v) and using (\ref{Eq_Krein_formula_3}) we obtain 
\begin{align*}
(1-B_2M(\lambda)B_1)B_2\Gamma_1f_B&=B_2\Gamma_1f_B-B_2M(\lambda)\Gamma_0f_B \\
&=B_2\Gamma_1f_B-B_2M(\lambda)\Gamma_0(f_B-f_0) \\
&=B_2\Gamma_1f_B-B_2\Gamma_1(f_B-f_0) \\
&=B_2\gamma(\overline{\lambda})^*h.
\end{align*}
Since $1-B_2M(\lambda)B_1$ is invertible according to (\ref{Eq_Step_3}) we conclude
\begin{equation*}
B_2\Gamma_1f_B=(1-B_2M(\lambda)B_1)^{-1}B_2\gamma(\overline{\lambda})^*h.
\end{equation*}
Using again $\Gamma_0(f_B-f_0)=B\Gamma_1f_B=B_1B_2\Gamma_1 f_B$ from (\ref{Eq_Krein_formula_3}) as well as (\ref{Eq_Krein_formula_2}) leads to
\begin{equation*}
f_B-f_0=\gamma(\lambda)B_1(1-B_2M(\lambda)B_1)^{-1}B_2\gamma(\overline{\lambda})^*h.
\end{equation*}
Now the Krein type resolvent formula (\ref{Eq_Krein_formula}) follows from \eqref{oho}.
\end{proof}

\section{An example: Laplacians on the half-space with singular Robin boundary conditions}
\label{sec3}

In this section we illustrate our abstract techniques from the previous section by applying Corollary~\ref{cor2_Selfadjoint_boundary_operators}
to an explicit boundary value problem. On the upper half-space $\dR^d_+=\{x\in\dR^d : x_d>0\}$ in $d\geq 2$ dimensions we consider the Laplacian with Robin boundary 
conditions $\tau_Nf=\alpha\tau_Df$ on $\partial\dR^d_+\simeq\dR^{d-1}$ involving an unbounded parameter function $\alpha:\dR^{d-1}\rightarrow\dR$. Here $\tau_D$ and $\tau_N$ denote the Dirichlet and Neumann trace operator, respectively.

In order to construct a suitable quasi boundary triple consider the operators
\begin{equation*}
 Tf=-\Delta f,\qquad\dom T=\bigl\{f\in H^{3/2}(\dR^d_+):\Delta f\in L^2(\dR^d_+)\bigr\},
\end{equation*}
and 
\begin{equation*}
 Sf=-\Delta f,\qquad\dom S=\bigl\{f\in H^{2}(\dR^d_+):\tau_D f=\tau_N f=0\bigr\},
\end{equation*}
as well as the boundary mappings
\begin{equation*}
 \Gamma_0 f=\tau_N f\quad\text{and}\quad\Gamma_1 f=\tau_D f,\qquad f\in\dom T.
\end{equation*}
The following proposition is essentially a consequence of the properties of the Dirichlet and Neumann trace operators and can be proved with standard techniques; cf. \cite[Proposition 4.6]{BL07}. The form of the Weyl function follows from \cite[(9.65)]{Grubb}.

\begin{prop}\label{prop_Quasi_boundary_triple_of_the_laplacian}
Let $S$, $T$, $\Gamma_0$ and $\Gamma_1$ be as above. 
Then $\{L^2(\dR^{d-1}),\Gamma_0,\Gamma_1\}$ is a quasi boundary triple for $T\subset S^*$ such that 
\begin{equation*}
\ran\Gamma_0=L^2(\dR^{d-1})\qquad\text{and}\qquad\ran\Gamma_1=H^1(\dR^{d-1}). 
\end{equation*}
Furthermore, $A_0=T\upharpoonright\ker\Gamma_0$ coincides with the Neumann Laplacian
\begin{equation*}
 A_Nf=-\Delta f,\qquad\dom A_N=\bigl\{f\in H^{2}(\dR^d_+):\tau_N f=0\bigr\},
\end{equation*}
and the corresponding Weyl function is given by
\begin{equation}\label{weyli}
M(\lambda)=(-\Delta_{\dR^{d-1}}-\lambda)^{-\frac{1}{2}},\qquad \lambda\in\dC\setminus[0,\infty),
\end{equation}
where $\Delta_{\dR^{d-1}}$ denotes the self-adjoint Laplacian in $L^2(\dR^{d-1})$ with domain $H^2(\dR^{d-1})$.
\end{prop}

It follows from $\dom\Delta_{\dR^{d-1}}^{\frac{1}{2}}=H^1(\dR^{d-1})$, the continuity of the natural embedding $H^1(\dR^{d-1})\hookrightarrow H^s(\dR^{d-1})$ for $s\leq 1$ and (\ref{weyli}) that 
\begin{equation}\label{referhere}
 M_{2,s}(\lambda):L^2(\dR^{d-1})\rightarrow H^s(\dR^{d-1}),\qquad\varphi\mapsto M_{2,s}(\lambda)\varphi\coloneqq M(\lambda)\varphi,
\end{equation}
is a bounded operator for every $s\leq 1$.
Moreover, the next lemma shows that the values $M(\lambda)$ 
of the Weyl function also induce densely defined and bounded operators from $L^p(\dR^{d-1})$ into $H^s(\dR^{d-1})$ 
for certain values of $p$ and $s$. This is essentially a consequence of the mapping properties of the resolvent of the Laplacian on $\dR^{d-1}$; 
for the convenience of the reader we provide
a short proof.

\begin{lem}\label{lem_Sobolev_Boundedness_of_Weyl_function}
Let $M$ be the Weyl function of the quasi boundary triple in Proposition~\ref{prop_Quasi_boundary_triple_of_the_laplacian}. 
For $\lambda\in\dC\setminus[0,\infty)$, $p\in[1,2)$ and $s<1-(d-1)(\frac{1}{p}-\frac{1}{2})$ the restriction
\begin{equation*}
 M_{p,s}(\lambda):L^p(\dR^{d-1})\rightarrow H^s(\dR^{d-1}),\qquad\varphi\mapsto M_{p,s}(\lambda)\varphi\coloneqq M(\lambda)\varphi, 
\end{equation*}
with $\dom M_{p,s}(\lambda)=L^p(\dR^{d-1})\cap L^2(\dR^{d-1})$ is a densely defined and bounded operator. 
\end{lem}

\begin{proof}
Denote by $\cF$ the Fourier transform in $L^2(\dR^{d-1})$. Then it follows from \eqref{weyli} that for every $\varphi\in L^2(\dR^{d-1})$ we get
\begin{equation*}
(\cF M(\lambda)\varphi)(\xi)=(|\xi|^2-\lambda)^{-\frac{1}{2}}(\cF\varphi)(\xi),\qquad\xi\in\dR^{d-1}.
\end{equation*}
Fix $r>0$ and choose a constant $C_r>0$ such that
\begin{equation*}
\frac{(1+|\xi|^2)^s}{||\xi|^2-\lambda|}\leq C_r\left\{\begin{array}{ll} 1, & \xi\in B_r\;, \\ |\xi|^{-(2-2s)}, & \xi\in\dR^{d-1}\setminus B_r\;, \end{array}\right.
\end{equation*}
where $B_r$ is the open ball with radius $r$ centered at $0$. Then for every function $\varphi\in L^p(\dR^{d-1})\cap L^2(\dR^{d-1})$ one has the estimate
\begin{equation}\label{Eq_Sobolev_Boundedness_of_Weyl_function_1}
\begin{split}
\Vert M(\lambda)\varphi\Vert^2_{H^s(\dR^{d-1})}&=\int_{\dR^{d-1}}(1+|\xi|^2)^s|(\cF M(\lambda)\varphi)(\xi)|^2d\xi\\
&=\int_{\dR^{d-1}}\frac{(1+|\xi|^2)^s}{||\xi|^2-\lambda|}|(\cF\varphi)(\xi)|^2d\xi\\
&\leq C_r\;\left(\int_{B_r}|(\cF\varphi)(\xi)|^2d\xi+\int_{\dR^{d-1}\setminus B_r}\frac{|(\cF\varphi)(\xi)|^2}{|\xi|^{2-2s}}d\xi\right).
\end{split}
\end{equation}
Using the H\"older inequality with the coefficients $\frac{p}{2-p}$ and $\frac{p}{2(p-1)}$ we obtain for the first integral
\begin{equation}\label{Eq_Sobolev_Boundedness_of_Weyl_function_5}
\int_{B_r}|(\cF\varphi)(\xi)|^2d\xi\leq|B_r|^{\frac{2-p}{p}}\Vert\cF\varphi\Vert^2_{L^{\frac{p}{p-1}}(\dR^{d-1})},
\end{equation}
and for the second integral 
\begin{equation}\label{Eq_Sobolev_Boundedness_of_Weyl_function_6}
\int_{\dR^{d-1}\setminus B_r}\frac{|(\cF\varphi)(\xi)|^2}{|\xi|^{2-2s}}d\xi\leq\left(\int_{\dR^{d-1}\setminus B_r}
|\xi|^{-\frac{(2-2s)p}{2-p}}d\xi\right)^{\frac{2-p}{p}}
\Vert\cF\varphi\Vert^2_{L^{\frac{p}{p-1}}(\dR^{d-1})}.
\end{equation}
As $s<1-(d-1)(\frac{1}{p}-\frac{1}{2})$ by assumption, we have $\frac{(2-2s)p}{2-p}>d-1$ and hence the integral on the right hand side of 
\eqref{Eq_Sobolev_Boundedness_of_Weyl_function_6} is finite. Furthermore, since the Fourier transform $\cF$ is bounded from $L^p(\dR^{d-1})$
into $L^{\frac{p}{p-1}}(\dR^{d-1})$ it follows from \eqref{Eq_Sobolev_Boundedness_of_Weyl_function_5} and \eqref{Eq_Sobolev_Boundedness_of_Weyl_function_6} 
that \eqref{Eq_Sobolev_Boundedness_of_Weyl_function_1} can finally be estimated by
\begin{equation*}
 \Vert M(\lambda)\varphi\Vert^2_{H^s(\dR^{d-1})}\leq C'\Vert\varphi\Vert^2_{L^p(\dR^{d-1})},\quad
 \varphi\in L^p(\dR^{d-1})\cap L^2(\dR^{d-1}),
\end{equation*}
with some constant $C'>0$. This completes the proof of Lemma~\ref{lem_Sobolev_Boundedness_of_Weyl_function}.
\end{proof}

The following lemma provides two important technical properties of 
the parameter function $\alpha$, which will be useful in the proof of Theorem~\ref{satz_Representation_via_boundary_values}.

\begin{lem}\label{lem_Properties_of_the_potential}
Let $\alpha\in L^p(\dR^{d-1})+L^\infty(\dR^{d-1})$ for some $p>2$. Then for every $t\in(0,1]$ one has
\begin{equation}\label{Eq_Lebesgue_regularity_of_the_power}
|\alpha|^t\in L^{\frac{p}{t}}(\dR^{d-1})+L^\infty(\dR^{d-1})
\end{equation}
and there exists a constant $C_\alpha>0$ such that
\begin{equation}\label{Eq_Boundedness_Hs}
\Vert|\alpha|^t\varphi\Vert_{L^2(\dR^{d-1})}\leq C_\alpha\Vert\varphi\Vert_{H^{\frac{t(d-1)}{p}}(\dR^{d-1})}
\end{equation}
holds for every $\varphi\in H^{\frac{t(d-1)}{p}}(\dR^{d-1})$.
\end{lem}

\begin{proof}
Decompose $\alpha=\alpha_p+\alpha_\infty$ for $\alpha_p\in L^p(\dR^{d-1})$ and $\alpha_\infty
\in L^\infty(\dR^{d-1})$ and define the functions
\begin{equation*}
\beta_{\frac{p}{t}}(x)=\left\{\begin{array}{ll} |\alpha(x)|^t, & x\in K, \\ 0, & x\notin K, \end{array}\right.
\quad\textnormal{and}\quad\beta_\infty(x)=\left\{\begin{array}{ll} 0, & x\in K, \\ |\alpha(x)|^t, & x\notin K, \end{array}\right.
\end{equation*}
where $K=\{x\in\dR^{d-1} : |\alpha(x)|>\Vert\alpha_\infty\Vert_{L^\infty(\dR^{d-1})}+1\}$. Note that $K$ is contained 
in the set $\{x\in\dR^{d-1} : |\alpha_p(x)|>1\}$, which has finite measure since $\alpha_p\in L^p(\dR^{d-1})$. Hence $K$ has finite measure as well. It is obvious that $\beta_\infty\in L^\infty(\dR^{d-1})$ and moreover, the estimate
\begin{align*}
\int_{\dR^{d-1}}|\beta_{\frac{p}{t}}(x)|^{\frac{p}{t}}dx&=\int_K|\alpha_p(x)+\alpha_\infty(x)|^p dx \\
&\leq 2^{p-1}\left(\int_K|\alpha_p(x)|^pdx+\int_K|\alpha_\infty(x)|^pdx\right) \\
&\leq 2^{p-1}\left(\Vert\alpha_p\Vert^p_{L^p(\dR^{d-1})}+|K|\;\Vert\alpha_\infty\Vert^p_{L^\infty(\dR^{d-1})}\right)
\end{align*}
shows that $\beta_{\frac{p}{t}}\in L^{\frac{p}{t}}(\dR^{d-1})$. Hence $|\alpha|^t=\beta_{\frac{p}{t}}+\beta_\infty\in L^{\frac{p}{t}}(\dR^{d-1})+L^\infty(\dR^{d-1})$.

\vspace{0.2cm}

Using the decomposition $|\alpha|^t=\beta_{\frac{p}{t}}+\beta_\infty$ from above, we can prove (\ref{Eq_Boundedness_Hs}) by estimating both terms separately. 
For the bounded part $\beta_\infty$ it is clear that
\begin{equation}\label{Eq_Regularity_properties_of_the_potential_1}
\begin{split}
\Vert\beta_\infty\varphi\Vert_{L^2(\dR^{d-1})}&\leq\Vert\beta_\infty\Vert_{L^\infty(\dR^{d-1})}\Vert\varphi\Vert_{L^2(\dR^{d-1})}\\
&\leq\Vert\beta_\infty\Vert_{L^\infty(\dR^{d-1})}\Vert\varphi\Vert_{H^{\frac{t(d-1)}{p}}(\dR^{d-1})}
\end{split}
\end{equation}
holds for all $\varphi\in H^{\frac{t(d-1)}{p}}(\dR^{d-1})$.
For the estimate of the unbounded part $\beta_{\frac{p}{t}}$ note first that by assumption we ensured $p>2\geq 2t$. Hence the H\"older inequality with the coefficients $\frac{p}{2t}$ and $\frac{p}{p-2t}$ yields
\begin{equation}\label{Eq_Regularity_properties_of_the_potential_2}
\begin{split}
\Vert\beta_{\frac{p}{t}}\varphi\Vert_{L^2(\dR^{d-1})}&\leq\Vert\beta_{\frac{p}{t}}\Vert_{L^{\frac{p}{t}}(\dR^{d-1})}\Vert\varphi\Vert_{L^{\frac{2p}{p-2t}}(\dR^{d-1})}\\
&\leq C\;\Vert\beta_{\frac{p}{t}}\Vert_{L^{\frac{p}{t}}(\dR^{d-1})}\Vert\varphi\Vert_{H^{\frac{t(d-1)}{p}}(\dR^{d-1})}
\end{split}
\end{equation}
for all $\varphi\in H^{\frac{t(d-1)}{p}}(\dR^{d-1})$, 
where $C>0$ is the constant of the Sobolev embedding theorem \cite[Theorem~8.12.4 Case I]{B12}. 
Combining (\ref{Eq_Regularity_properties_of_the_potential_1}) and (\ref{Eq_Regularity_properties_of_the_potential_2}) leads to the estimate (\ref{Eq_Boundedness_Hs}).
\end{proof}

In the next lemma we recall a simple estimate for functions $f\in H^1(\dR^d_+)$. For the convenience of the reader we provide a short proof.

\begin{lem}\label{lem_Epsilon_estimate} 
Let $s\in[0,1)$. Then for every $\varepsilon>0$ there exists a constant $C_\varepsilon>0$ such that
\begin{equation}\label{ok33}
\Vert f\Vert_{H^s(\dR^d_+)}^2\leq\varepsilon\Vert\nabla f\Vert^2_{L^2(\dR^d_+,\dC^d)}+C_\varepsilon\Vert f\Vert^2_{L^2(\dR^d_+)}
\end{equation}
holds for every $f\in H^1(\dR^d_+)$.
\end{lem}

\begin{proof}
Recall from \cite[\textsection 3~Theorem 5]{S70} that there exists an extension operator 
$E:L^2(\dR^d_+)\rightarrow L^2(\dR^d)$ which satisfies
\begin{equation}\label{eee}
\Vert Eg\Vert_{L^2(\dR^d)}\leq c_E\Vert g\Vert_{L^2(\dR^d_+)}\quad\textnormal{and}\quad
\Vert Ef\Vert_{H^1(\dR^d)}\leq c_E\Vert f\Vert_{H^1(\dR^d_+)}
\end{equation}
for some $c_E>0$ and all $g\in L^2(\dR^d_+)$, $f\in H^1(\dR^d_+)$. From \cite[Theorem~3.30]{HT08} we can conclude that for $\varepsilon'>0$ there exists $C_{\varepsilon'}>0$ such that
\begin{equation*}
\Vert f\Vert_{H^s(\dR^d_+)}\leq\Vert Ef\Vert_{H^s(\dR^d)}
\leq \varepsilon'\Vert Ef\Vert_{H^1(\dR^d)}+C_{\varepsilon'}\Vert Ef\Vert_{L^2(\dR^d)}
\end{equation*}
for every $f\in H^1(\dR^d_+)$. Together with \eqref{eee} this leads to \eqref{ok33}.
\end{proof}

After these preparations we are now ready to formulate and 
prove the main theorem of this section.

\begin{thm}\label{satz_Representation_via_boundary_values}
Let $\alpha\in L^p(\dR^{d-1})+L^\infty(\dR^{d-1})$ be a real-valued function and assume that 
$p>\frac{4}{3}(d-1)$ if $d\geq 3$ and $p>2$ if $d=2$. Then the Robin-Laplacian
\begin{equation}\label{Eq_Robin_Laplacian_QBT}
A_\alpha f=-\Delta f,\qquad\dom A_\alpha=\left\{f\in H^{3/2}(\dR^d_+) : \begin{array}{l} \Delta f\in L^2(\dR^d_+), \\ 
\alpha\tau_Df=\tau_Nf \end{array}\right\},
\end{equation}
is self-adjoint in $L^2(\dR^d_+)$ and for every $\lambda\in\rho(A_\alpha)\setminus[0,\infty)$ the Krein type resolvent formula
\begin{equation*}
\begin{split}
&(A_\alpha-\lambda)^{-1}-(A_N-\lambda)^{-1} \\
&\hspace{0.5cm}=\gamma(\lambda)\sig(\alpha)|\alpha|^{\frac{1}{3}}\bigl(1+|\alpha|^{\frac{2}{3}}(-\Delta_{\dR^{d-1}}-\lambda)^{-\frac{1}{2}}\sig(\alpha)|\alpha|^{\frac{1}{3}}\bigr)^{-1}
|\alpha|^{\frac{2}{3}}\gamma(\overline{\lambda})^*
\end{split}
\end{equation*}
is valid.
\end{thm}

\begin{proof}
This theorem is a consequence of Corollary \ref{cor2_Selfadjoint_boundary_operators} and hence in the following it will be shown that its assumptions are satisfied. We start by defining
the multiplication operator
\begin{equation*}
B\varphi=\alpha\varphi,\qquad\dom B=H^1(\dR^{d-1}),
\end{equation*}
in the boundary space $L^2(\dR^{d-1})$.
Note that by assumption we have $p>2$ as well as $p>t(d-1)$, for every $t\in(0,1]$ in any dimension $d\geq 2$. Hence by Lemma \ref{lem_Properties_of_the_potential} the estimate
\begin{equation}\label{Eq_Representation_via_boundary_values_4}
\Vert|\alpha|^t\varphi\Vert_{L^2(\dR^{d-1})}\leq C_\alpha\Vert\varphi\Vert_{H^{\frac{t(d-1)}{p}}(\dR^{d-1})}\leq C_\alpha\Vert\varphi\Vert_{H^1(\dR^{d-1})}
\end{equation}
holds for every $\varphi\in H^1(\dR^{d-1})$ and the operator $B$ is well-defined. Clearly the first inequality in 
\eqref{Eq_Representation_via_boundary_values_4} also holds for $\varphi\in H^{\frac{t(d-1)}{p}}(\mathbb{R}^{d-1})$. Next we decompose $B$ into
\begin{alignat*}{2}
B_1\varphi&=\sig(\alpha)|\alpha|^{\frac{1}{3}}\varphi,\quad&&\dom B_1=\bigl\{\varphi\in L^2(\dR^{d-1}) : \vert\alpha\vert^{\frac{1}{3}}\varphi\in L^2(\dR^{d-1})\bigr\}, \\
B_2\varphi&=|\alpha|^{\frac{2}{3}}\varphi,&&\dom B_2=\bigl\{\varphi\in L^2(\dR^{d-1}) : \vert\alpha\vert^{\frac{2}{3}}\varphi\in L^2(\dR^{d-1})\bigr\}.
\end{alignat*}
Using the first estimate in (\ref{Eq_Representation_via_boundary_values_4}) it follows that every $\varphi\in\dom B=H^1(\dR^{d-1})$ satisfies 
$\varphi\in\dom B_2$ and $B_2\varphi\in\dom B_1$, and hence the operator inclusion $B\subset B_1B_2$ holds.

For the operators $B$, $B_1$ and $B_2$ we now verify
the assumptions in Corollary~\ref{cor2_Selfadjoint_boundary_operators}. First of all, 
since $\alpha$ is real-valued, it is clear that the operator $B$ is symmetric in $L^2(\dR^{d-1})$. 
Moreover, $\ran\Gamma_0=L^2(\dR^{d-1})$ as well as $\ran\Gamma_1=H^1(\dR^{d-1})$ holds by 
Proposition \ref{prop_Quasi_boundary_triple_of_the_laplacian} and hence also assumption (iii) in Corollary~\ref{cor2_Selfadjoint_boundary_operators} is fulfilled. 
Therefore, it remains to choose a suitable $\lambda_0\in\rho(A_N)\cap\dR=(-\infty,0)$ such that the assumptions (i) and (ii) are satisfied.

Using again (\ref{Eq_Representation_via_boundary_values_4}), the boundedness of the Dirichlet trace operator 
$$\tau_D:H^{\frac{2(d-1)}{3p}+\frac{1}{2}}(\dR^d_+)\rightarrow H^{\frac{2(d-1)}{3p}}(\dR^{d-1}),$$ 
and Lemma~\ref{lem_Epsilon_estimate}, we find a constant $c_1>0$
such that 
\begin{equation}\label{Eq_Representation_via_boundary_values_2}
\begin{split}
\Vert|\alpha|^{\frac{2}{3}}\tau_Dg\Vert^2_{L^2(\dR^{d-1})}&\leq C_\alpha^2\Vert\tau_Dg\Vert^2_{H^{\frac{2(d-1)}{3p}}(\dR^{d-1})} \\
&\leq C_\alpha^2\Vert\tau_D\Vert^2\Vert g\Vert^2_{H^{\frac{2(d-1)}{3p}+\frac{1}{2}}(\dR^d_+)} \\
&\leq\frac{1}{2}\Vert\nabla g\Vert^2_{L^2(\dR^d_+,\dC^d)}+c_1\Vert g\Vert^2_{L^2(\dR^d_+)}
\end{split}
\end{equation}
holds for all $g\in H^1(\dR^d_+)$. In the last step it was crucial that $\frac{2(d-1)}{3p}+\frac{1}{2}<1$, which is equivalent to $p>\frac{4}{3}(d-1)$ and is fulfilled by the assumptions on $p$ in every dimension $d\geq 2$.
In the same way we find a constant $c_2>0$ such that
\begin{equation}\label{Eq_Representation_via_boundary_values_5}
\Vert|\alpha|^{\frac{1}{3}}\tau_Dg\Vert^2_{L^2(\dR^{d-1})}\leq\frac{1}{2}\Vert\nabla g\Vert^2_{L^2(\dR^d_+,\dC^d)}+c_2\Vert g\Vert^2_{L^2(\dR^d_+)}
\end{equation}
holds for all $g\in H^1(\dR^d_+)$.
For the choice $\lambda_0\coloneqq-2\max\{c_1,c_2\}\in\rho(A_N)$, the estimates (\ref{Eq_Representation_via_boundary_values_2}) and (\ref{Eq_Representation_via_boundary_values_5}) turn into
\begin{align}
&\Vert|\alpha|^{\frac{2}{3}}\tau_Dg\Vert^2_{L^2(\dR^{d-1})}\leq\frac{1}{2}
\bigl(\Vert\nabla g\Vert^2_{L^2(\dR^d_+,\dC^d)}-\lambda_0\Vert g\Vert^2_{L^2(\dR^d_+)}\bigr), \label{Eq_Representation_via_boundary_values_8} \\
&\Vert|\alpha|^{\frac{1}{3}}\tau_Dg\Vert^2_{L^2(\dR^{d-1})}\leq\frac{1}{2}
\bigl(\Vert\nabla g\Vert^2_{L^2(\dR^d_+,\dC^d)}-\lambda_0\Vert g\Vert^2_{L^2(\dR^d_+)}\bigr), \label{Eq_Representation_via_boundary_values_9}
\end{align}
for all $g\in H^1(\dR^d_+)$.

\vspace{0.2cm}
\noindent
\textit{Assumption (ii).} In order to check $\ran(B_2\overline{M(\lambda_0)B_1})\subset\dom B_1$ we have to show 
$|\alpha|^{\frac{1}{3}}B_2\overline{M(\lambda_0)B_1}\varphi=\vert\alpha\vert\overline{M(\lambda_0)B_1}\varphi\in L^2(\dR^{d-1})$ 
for all functions $\varphi\in\dom(B_2\overline{M(\lambda_0)B_1})$. Using (\ref{Eq_Representation_via_boundary_values_4}) it suffices to verify $\overline{M(\lambda_0)B_1}\varphi\in H^{\frac{d-1}{p}}(\dR^{d-1})$.

First consider $\varphi\in\dom(M(\lambda_0)B_1)$ and choose $\beta_{3p}\in L^{3p}(\dR^{d-1})$ and $\beta_\infty\in L^\infty(\dR^{d-1})$ 
such that $\sig(\alpha)|\alpha|^{\frac{1}{3}}=\beta_{3p}+\beta_\infty$; cf. \eqref{Eq_Lebesgue_regularity_of_the_power}. Then by the boundedness of the
Weyl function in Lemma \ref{lem_Sobolev_Boundedness_of_Weyl_function} and \eqref{referhere} we obtain
\begin{equation}\label{Eq_Representation_via_boundary_values_3}
\begin{split}
\Vert M(\lambda_0)&B_1\varphi\Vert_{H^{\frac{d-1}{p}}(\dR^{d-1})} \\
&\leq\Vert M(\lambda_0)\beta_{3p}\varphi\Vert_{H^{\frac{d-1}{p}}(\dR^{d-1})}+\Vert M(\lambda_0)\beta_\infty\varphi\Vert_{H^{\frac{d-1}{p}}(\dR^{d-1})} \\
&\leq c\left(\Vert\beta_{3p}\varphi\Vert_{L^{\frac{6p}{3p+2}}(\dR^{d-1})}+\Vert\beta_\infty\varphi\Vert_{L^2(\dR^{d-1})}\right) \\
&\leq c\left(\Vert\beta_{3p}\Vert_{L^{3p}(\dR^{d-1})}+\Vert\beta_\infty\Vert_{L^\infty(\dR^{d-1})}\right)\Vert\varphi\Vert_{L^2(\dR^{d-1})},
\end{split}
\end{equation}
where Lemma~\ref{lem_Sobolev_Boundedness_of_Weyl_function} was used in the penultimate inequality with $s$ and $p$ replaced by $\tfrac{d-1}{p}$ and $\tfrac{6p}{3p+2}$, respectively, which is possible since $p>\tfrac{4}{3}(d-1)$ holds by assumption for every dimension $d\geq 2$. Furthermore, in the last estimate the H\"older inequality with the exponents $\frac{3p+2}{2}$ and $\frac{3p+2}{3p}$ was used.

Now let $\varphi\in\dom(\overline{M(\lambda_0)B_1})$ and pick a sequence $(\varphi_n)\subset\dom(M(\lambda_0)B_1)$ 
such that $\varphi_n\rightarrow\varphi$ and $M(\lambda_0)B_1\varphi_n\rightarrow\overline{M(\lambda_0)B_1}\varphi$ for $n\rightarrow\infty$ in $L^2(\dR^{d-1})$. 
It is clear from (\ref{Eq_Representation_via_boundary_values_3}) that the sequence $(M(\lambda_0)B_1\varphi_n)$ converges in $H^{\frac{d-1}{p}}(\dR^{d-1})$
to an element $g\in H^{\frac{d-1}{p}}(\dR^{d-1})$. Hence it follows that 
$$g=\overline{M(\lambda_0)B_1}\varphi\in H^{\frac{d-1}{p}}(\dR^{d-1}).$$ 
Therefore, assumption (ii) in Corollary~\ref{cor2_Selfadjoint_boundary_operators} holds.

\vspace{0.2cm}
\noindent
\textit{Assumption (i).} We prove $1\in\rho(B_2\overline{M(\lambda_0)B_1})$ by showing that $B_2\overline{M(\lambda_0)B_1}$ is an everywhere defined bounded 
operator with norm strictly less than $1$.

For this we define the inner product
\begin{equation*}
(f,g)_{\lambda_0}\coloneqq(\nabla f,\nabla g)_{L^2(\dR^d_+,\dC^d)}-\lambda_0(f,g)_{L^2(\dR^d_+)},\quad f,g\in H^1(\dR^d_+),
\end{equation*}
and note that the corresponding norm is equivalent to the usual $H^1(\dR^d_+)$-norm. 
Fix now any $\varphi\in\dom(B_2M(\lambda_0)B_1)$ and use (\ref{Eq_Representation_via_boundary_values_8}) for $g=\gamma(\lambda_0)B_1\varphi$ to obtain the estimate
\begin{equation*}
\begin{split}
\Vert B_2M(\lambda_0)B_1\varphi\Vert^2_{L^2(\dR^{d-1})}&=\Vert|\alpha|^{\frac{2}{3}}\tau_D\gamma(\lambda_0)B_1\varphi\Vert^2_{L^2(\dR^{d-1})} \\
&\leq\frac{1}{2}\Vert\gamma(\lambda_0)B_1\varphi\Vert^2_{\lambda_0} \\
&=\frac{1}{2}\sup\limits_{h\in H^1(\dR^d_+)\setminus\{0\}}\frac{(\gamma(\lambda_0)B_1\varphi,h)^2_{\lambda_0}}{\Vert h\Vert^2_{\lambda_0}}.
\end{split}
\end{equation*}
Using the first Green identity and the properties
\begin{equation*}
(-\Delta-\lambda_0)\gamma(\lambda_0)B_1\varphi=0\quad\textnormal{and}\quad\tau_N\gamma(\lambda_0)B_1\varphi=B_1\varphi,
\end{equation*}
of the $\gamma$-field, which follow immediately from its definition (\ref{Eq_Gamma_field}) and Proposition~\ref{prop_Quasi_boundary_triple_of_the_laplacian}, we find
\begin{equation*}
\begin{split}
(\gamma(\lambda_0)B_1\varphi,h)_{\lambda_0}&=(\nabla\gamma(\lambda_0)B_1\varphi,\nabla h)_{L^2(\dR^d_+,\dC^d)} - \lambda_0 (\gamma(\lambda_0)B_1 \varphi,h)_{L^2(\dR^d_+)}\\
&=(\nabla\gamma(\lambda_0)B_1\varphi,\nabla h)_{L^2(\dR^d_+,\dC^d)} +(\Delta \gamma(\lambda_0) B_1 \varphi,h)_{L^2(\dR^d_+)}\\
&=(\tau_N\gamma(\lambda_0)B_1\varphi,\tau_D h)_{L^2(\dR^{d-1})}\\
&=(B_1\varphi,\tau_D h)_{L^2(\dR^{d-1})}
\end{split}
\end{equation*}
and hence
\begin{align*}
\Vert B_2M(\lambda_0)B_1\varphi\Vert^2_{L^2(\dR^{d-1})}&\leq\frac{1}{2}\sup\limits_{h\in H^1(\dR^d_+)\setminus\{0\}}
\frac{(B_1\varphi,\tau_D h)^2_{L^2(\dR^{d-1})}}{\Vert h\Vert^2_{\lambda_0}} \\
&\leq\frac{1}{2}\Vert\varphi\Vert^2_{L^2(\dR^{d-1})}\sup\limits_{h\in H^1(\dR^d_+)\setminus\{0\}}\frac{\Vert|\alpha|^{\frac{1}{3}}\tau_Dh\Vert^2_{L^2(\dR^{d-1})}}
{\Vert h\Vert_{\lambda_0}^2}.
\end{align*}
Equation (\ref{Eq_Representation_via_boundary_values_9}) then leads to the estimate
\begin{equation}\label{Eq_Representation_via_boundary_values_1}
\Vert B_2M(\lambda_0)B_1\varphi\Vert^2_{L^2(\dR^{d-1})}\leq\frac{1}{4}\Vert\varphi\Vert^2_{L^2(\dR^{d-1})}
\end{equation}
for any $\varphi\in\dom(B_2M(\lambda_0)B_1)$.

\vspace{0.2cm}

As $B_2$ is closed and \eqref{Eq_Representation_via_boundary_values_3} implies that $M(\lambda_0)B_1$ is bounded in $L^2(\dR^{d-1})$ it follows that $B_2\overline{M(\lambda_0)B_1}$ is closed in $L^2(\dR^{d-1})$ as well. Since $\overline{B_2M(\lambda_0)B_1}$ is everywhere defined, this however implies $\overline{B_2M(\lambda_0)B_1}=B_2\overline{M(\lambda_0)B_1}$ and hence $1\in\rho(B_2\overline{M(\lambda_0)B_1})$ follows from (\ref{Eq_Representation_via_boundary_values_1}). This completes the proof of Theorem~\ref{satz_Representation_via_boundary_values}.
\end{proof}

\begin{rem}
If one uses Corollary \ref{cor_Selfadjoint_boundary_operators} instead of Corollary \ref{cor2_Selfadjoint_boundary_operators} 
in the proof of Theorem \ref{satz_Representation_via_boundary_values} only 
$\alpha\in L^p(\dR^{d-1})+L^\infty(\dR^{d-1})$ with $p>2(d-1)$ can be treated. In fact, in this situation one chooses $B_2=B$ to be the multiplication operator with $\alpha$
and for the estimate \eqref{Eq_Representation_via_boundary_values_2} (with $\alpha$ instead of $\vert\alpha\vert^\frac{2}{3}$) it is necessary to restrict 
to $p>2(d-1)$. Thus, for Laplacians with singular Robin boundary conditions  Theorem \ref{satz_Selfadjoint_boundary_operators} 
and Corollary~\ref{cor2_Selfadjoint_boundary_operators} allow a larger class of boundary parameters $\alpha$ than Corollary \ref{cor_Selfadjoint_boundary_operators}.
\end{rem}

\begin{rem}\label{agwrem}
A variant of Theorem~\ref{satz_Representation_via_boundary_values} for more general elliptic second order operators 
on a certain class of unbounded non-smooth domains with Robin boundary conditions containing also differential or pseudodifferential operators 
can be found in \cite{AGW14}. In our situation for a Robin Laplacian on $\dR^d_+$ with an $H^{1/2}$-smooth real-valued 
$$\alpha\in H^{1/2}_p(\mathbb{R}^{d-1}),\qquad p>2(d-1),$$
it follows from \cite[Theorem 7.2]{AGW14} that the operator
\begin{equation*}
A_\alpha f=-\Delta f,\quad\dom A_\alpha=\left\{f\in H^2(\dR^d_+) : \alpha\tau_Df=\tau_Nf \right\},
\end{equation*}
is self-adjoint in $L^2(\dR^d_+)$.
\end{rem}

Self-adjoint Laplacians with Robin boundary conditions can also be defined via the densely defined, symmetric form 
\begin{equation}\label{Eq_Quadratic_form}
\mathfrak a_\alpha[f]=\Vert\nabla f\Vert^2_{L^2(\dR^d,\dC^d)}-\int_{\dR^{d-1}}\alpha|\tau_Df|^2dx,\quad\dom\mathfrak a_\alpha=H^1(\dR^d_+),
\end{equation}
and the first representation theorem \cite[VI~Theorem~2.1]{K76}. The following proposition shows that 
this method allows a larger class of boundary parameters $\alpha$ as 
Theorem \ref{satz_Representation_via_boundary_values} does, but leads to an operator $A_\alpha$ with a less regular operator domain.
However, for functions $\alpha$ satisfying the stronger assumptions in Theorem \ref{satz_Representation_via_boundary_values}, 
the operators in \eqref{Eq_Robin_Laplacian_form} below and in \eqref{Eq_Robin_Laplacian_QBT} coincide. 
A variant of Proposition~\ref{prop_Operator_via_Quadratic_form} for bounded Lipschitz
domains can be found in \cite[Theorem 4.5 and Lemma 5.3]{GM09}.

\begin{prop}\label{prop_Operator_via_Quadratic_form}
Let $\alpha\in L^p(\dR^{d-1})+L^\infty(\dR^{d-1})$ be a real-valued function for $p=d-1$ if $d\geq 3$ and $p>1$ if $d=2$. 
Then the quadratic form $\mathfrak a_\alpha$ in \eqref{Eq_Quadratic_form} is semibounded and closed. The corresponding self-adjoint operator in $L^2(\dR^d_+)$ 
is given by
\begin{equation}\label{Eq_Robin_Laplacian_form}
A_\alpha f=-\Delta f,\quad\dom A_\alpha=\left\{f\in H^1(\dR^d_+) : \begin{array}{l} \Delta f\in L^2(\dR^d_+), \\ \alpha\tau_Df=\tau_Nf \end{array}\right\}.
\end{equation}
\end{prop}

\begin{proof}
In order to prove that the form $\mathfrak a_\alpha$ is semibounded and closed we split $\mathfrak a_\alpha$ into the two quadratic forms
\begin{equation*}
\mathfrak a[f]=\Vert\nabla f\Vert^2_{L^2(\dR^d_+,\dC^d)}\quad\text{and}\quad
\mathfrak t[f]=\int_{\dR^{d-1}}\alpha(x)|\tau_Df(x)|^2dx
\end{equation*}
with $\dom\mathfrak a=\dom\mathfrak t= H^1(\dR^d_+)$
and observe that $\mathfrak a$ is a densely defined, nonnegative, closed form in $L^2(\dR^{d-1})$. Now it suffices to check that $\mathfrak t$ is relatively bounded with respect 
to $\mathfrak a$ with relative bound $<1$, that is, for some $a\geq 0$ and $0\leq b<1$ 
\begin{equation}\label{Eq_Form_method_3}
\vert\mathfrak t[f]\vert\leq a \Vert f\Vert^2_{L^2(\dR^d_+)} + b\,  \mathfrak a[f],\qquad f\in H^1(\dR^d_+),
\end{equation}
since in this case the KLMN theorem \cite[Theorem 6.24]{Teschl} 
(see also \cite[VI~Theorem~1.33]{K76}) yields that the form $\mathfrak a_\alpha=\mathfrak a-\mathfrak t$ is densely defined, 
closed, and semibounded in $L^2(\dR^{d-1})$.
To verify \eqref{Eq_Form_method_3}, decompose the function $\alpha=\alpha_p+\alpha_\infty$ in the sum of 
$\alpha_p\in L^p(\dR^{d-1})$ and $\alpha_\infty\in L^\infty(\dR^{d-1})$, and let $\varepsilon>0$.
For the unbounded part $\alpha_p$ choose a sufficiently large $\gamma_\varepsilon>0$ such that
\begin{equation}\label{Eq_Form_method_4}
\Vert\alpha_p\Vert_{L^p(K_\varepsilon)}\leq\varepsilon\quad\text{with}\quad K_\varepsilon=\bigl\{x\in\dR^{d-1} : |\alpha_p(x)|>\gamma_\varepsilon\bigr\},
\end{equation}
and write $\alpha_p$ as the sum of 
\begin{equation*}
\alpha_p^{(0)}(x)=\left\{\begin{array}{ll} 0, & x\in K_\varepsilon, \\ \alpha_p(x), & x\notin K_\varepsilon, \end{array}\right.\quad\text{and}\quad\alpha_p^{(1)}(x)=\left\{\begin{array}{ll} \alpha_p(x), & x\in K_\varepsilon, \\ 0, & x\notin K_\varepsilon. \end{array}\right.
\end{equation*}
With this decomposition we now estimate the form $\mathfrak t$ by
\begin{equation}\label{Eq_Form_method_1}
\begin{split}
|\mathfrak t[f]|\leq\int_{\dR^{d-1}}|&\alpha_\infty(x)+\alpha_p^{(0)}(x)|\;|\tau_Df(x)|^2dx\\
&\qquad\qquad+\int_{\dR^{d-1}}|\alpha_p^{(1)}(x)|\;|\tau_Df(x)|^2dx
\end{split}
\end{equation}
and discuss both integrals on the right hand side of \eqref{Eq_Form_method_1} separately. For the first integral we fix some arbitrary $s\in(\frac{1}{2},1)$ and use the continuity of 
$\tau_D:H^s(\dR^d_+)\rightarrow H^{s-\frac{1}{2}}(\dR^{d-1})$ as well as Lemma~\ref{lem_Epsilon_estimate} to obtain
\begin{equation}\label{Eq_Form_method_5}
\begin{split}
\int_{\dR^{d-1}}|\alpha_\infty(x)+&\alpha_p^{(0)}(x)|\;|\tau_Df(x)|^2dx\\
&\quad\leq(\Vert\alpha_\infty\Vert_{L^\infty(\dR^{d-1})}+\gamma_\varepsilon)\Vert\tau_Df\Vert^2_{L^2(\dR^{d-1})} \\
&\quad\leq c'\Vert  f\Vert^2_{H^s(\dR^d_+)} \\
&\quad\leq c'\varepsilon\Vert\nabla f\Vert^2_{L^2(\dR^d_+,\dC^d)}+c'C_\varepsilon\Vert f\Vert^2_{L^2(\dR^d_+)},
\end{split}
\end{equation}
where $c'=(\Vert\alpha_\infty\Vert_{L^\infty(\dR^{d-1})}+\gamma_\varepsilon)\Vert\tau_D\Vert^2$ and $C_\varepsilon$ is the constant in Lemma~\ref{lem_Epsilon_estimate}.
For the estimate of the second integral in \eqref{Eq_Form_method_1} we first use the H\"older inequality and \eqref{Eq_Form_method_4} to obtain
\begin{equation*}
\begin{split}
\int_{\dR^{d-1}}|\alpha_p^{(1)}(x)|\;|\tau_Df(x)|^2dx&\leq\Vert\alpha_p^{(1)}\Vert_{L^p(\dR^{d-1})}\Vert\tau_Df\Vert^2_{L^{\frac{2p}{p-1}}(\dR^{d-1})} \\
&\leq\varepsilon\Vert \tau_D f\Vert_{L^{\frac{2p}{p-1}}(\mathbb{R}^{d-1})}^2.
\end{split}
\end{equation*}
By the given assumptions on $p$ we can now apply the Sobolev embedding theorem \cite[Theorem~8.12.4 Case I]{B12} if $d\geq 3$ 
and \cite[Theorem~8.12.4 Case II]{B12} if $d=2$. This leads to the estimate
\begin{equation}\label{Eq_Form_method_6}
\int_{\dR^{d-1}}|\alpha_p^{(1)}|\;|\tau_Df(x)|^2dx\leq\varepsilon c''\Vert \tau_D f\Vert_{H^{\frac{1}{2}}(\dR^{d-1})}^2\leq\varepsilon c'''\Vert f\Vert^2_{H^1(\dR^d_+)}
\end{equation}
with some constants $c'',c'''>0$. From \eqref{Eq_Form_method_5} and \eqref{Eq_Form_method_6} we conclude that \eqref{Eq_Form_method_3} holds for all $b>0$ and hence it follows, in particular,
that $\mathfrak a_\alpha$ closed and semibounded.

\vspace{0.2cm}

We leave it to the reader to verify that 
 the self-adjoint operator corresponding to $\mathfrak a_\alpha$ is given by \eqref{Eq_Robin_Laplacian_form}.
\end{proof}

\end{document}